\documentclass{amsart}

\usepackage{amscd, amssymb, amsmath, amsthm}
\usepackage{enumerate, latexsym, mathrsfs}
\usepackage{xypic}
\usepackage[
		bookmarks=true, bookmarksopen=true,%
    bookmarksdepth=3,bookmarksopenlevel=2,%
    colorlinks=true,%
    linkcolor=blue,%
    citecolor=blue]{hyperref}
\newtheorem{theorem}{Theorem}[section]
\newtheorem{lemma}[theorem]{Lemma}

\theoremstyle{definition}

\newtheorem{notation}[theorem]{Notation}

\newtheorem{example}[theorem]{Example}

\newtheorem{remark}[theorem]{Remark}

\numberwithin{equation}{section}

\newcommand{\Real}{{\mathbb R}}
\newcommand{\Rational}{{\mathbb Q}}
\newcommand{\Complex}{{\mathbb C}}
\newcommand{\Integral}{{\mathbb Z}}

\title{On Hempel pairs and Turaev--Viro invariants}
     
\author[Yi Liu]{%
        Yi Liu} 
\address{%
        Beijing International Center for Mathematical Research, Peking University\\
				Beijing 100871, China P.R.} 
\email{%
    liuyi@bicmr.pku.edu.cn}

\thanks{Partially supported by NSFC Grant 11925101, 
and National Key R\&D Program of China 2020YFA0712800}
\subjclass[2020]{Primary 57K31; Secondary 57K16}
\keywords{profinite completion, mapping class group, Turaev--Viro invariant}

\date{%
 \today} 

\begin{document}

\begin{abstract}
	Surface bundles arising from periodic mapping classes 
	may sometimes have non-isomorphic, 
	but profinitely isomorphic fundamental groups.
	Pairs of this kind have been discovered by Hempel.
	This paper exhibits examples of nontrivial Hempel pairs
	where the mapping tori can be distinguished by some Turaev--Viro invariants,
	and also examples 
	where they cannot be distinguished by any Turaev--Viro invariants.
\end{abstract}

\maketitle

\section{Introduction}

Let $S$ be a connected closed orientable surface.
Denote by $\mathrm{Mod}(S)$ the mapping class group of $S$,
whose elements are the isotopy classes of orientation-preserving self-homeomorphisms of $S$.

A \emph{Hempel pair} as we call
refers to a pair of periodic mapping classes $[f_A],[f_B]\in\mathrm{Mod}(S)$
of identical order $d\geq1$,
such that $[f_B]=[f_A^k]$ holds for some integer $k$ coprime to $d$.
Hence $[f_A]=[f_B^{k^*}]$ holds for any congruence inverse $k^*$ of $k$ modulo $d$,
(that is, $k^*k\equiv1\bmod d$).
Hempel studied such pairs \cite{Hempel_quotient},
finding out that 
the fundamental groups of their mapping tori $M_A$ and $M_B$ 
always have identical collections of (isomorphism types of) finite quotient groups.
This is equivalent to saying that the profinite completions of $\pi_1(M_A)$ and $\pi_1(M_B)$
are isomorphic groups.
Hempel discovered examples of such pairs
with non-isomorphic fundamental groups.
This is equivalent to saying that $M_A$ and $M_B$ are not homeomorphic $3$--manifolds.

We call $[f_A]$ and $[f_B]$ a \emph{nontrivial} Hempel pair,
if $M_A$ and $M_B$ are not homeomorphic.
There are no nontrivial Hempel pairs when $S$ is a sphere or a torus,
as the condition forces $d=1$ or $d\in \{1,2,3,4,6\}$, and hence $k=\pm1$.
One may obtain a nontrivial Hempel pair of order $5$ when $S$ has genus $2$.
Nontrivial Hempel pairs are a source of distinct $3$--manifold pairs
that cannot be distinguished by their profinite fundamental groups.
Among $3$--manifolds,
the question as to which topological invariants are determined by
the profinite fundamental group has stimulated a lot of fruitful study 
in recent years. See \cite{Reid_survey} and \cite[Section 9]{LS_survey}
for past surveys on that fast-growing topic.

Turaev--Viro invariants are topological invariants of closed $3$--manifolds,
originally constructed using quantum 6j-symbols \cite{TV_invariant}.
That they are generally \emph{not} profinite invariants
is evident from their explicit values on all lens spaces \cite{Sokolov_lens}.
Meanwhile, there are many homeomorphically distinct 
torus bundles over a circle, 
whose monodromies are Anosov and mutually conjugate up to inverse 
in every congruence quotient of $\mathrm{Mod}(T^2)\cong\mathrm{SL}(2,\Integral)$.
These torus bundles have isomorphic profinite fundamental groups.
Funar shows that no Turaev--Viro invariants 
(associated to any spherical fusion categories) distinguish these torus bundles
\cite[Proposition 1.1]{Funar_torus_bundles}.

In this paper, we take up the question
as to whether Turaev--Viro invariants distinguish nontrivial Hempel pairs.
We shall content ourselves 
with the $\mathrm{SU}(2)$ and the $\mathrm{SO}(3)$ Turaev--Viro invariants,
as they are mentioned the most often.
The $\mathrm{SU}(2)$ series can be fully listed as
$\mathrm{TV}_{r,s}$, for any integer $r\geq3$ and any integer $s$ coprime to $r$;
the $\mathrm{SO}(3)$ series can be fully listed as
$\mathrm{TV}'_{r,s}$, for any odd integer $r\geq3$ 
and any even integer $s$ coprime to $r$.
More economically, one could focus on 
$\mathrm{TV}_{r,1}$ ($r$ even), and $\mathrm{TV}'_{r,r-1}$ ($r$ odd),
together with $\mathrm{TV}_{3,1}$ and $\mathrm{TV}_{3,2}$,
which depend only on the $\Integral/2\Integral$ cohomology ring.
These determine all the other ones. 
See Section \ref{Sec-tv} for the notations and more review.

Our conclusion can be summarized as follows.

\begin{theorem}\label{main_hempel}\
\begin{enumerate}
\item For any integer $d\geq5$ other than $6$, 
there exists some nontrivial Hempel pair of order $d$,
such that the mapping tori can be distinguished by 
some $\mathrm{SU}(2)$ Turaev--Viro invariant,
and if $d$ is odd, also by some $\mathrm{SO}(3)$ Turaev--Viro invariant.
\item For any prime integer $p\geq5$, 
there exists some nontrivial Hempel pair of order $p$,
such that the mapping tori cannot be distinguished 
by any $\mathrm{SU}(2)$ or $\mathrm{SO}(3)$ Turaev--Viro invariants.
\end{enumerate}
\end{theorem}

Theorem \ref{main_hempel} is proved in Section \ref{Sec-examples} by 
exhibiting concrete families of examples.
Our simplest distinguishable Hempel pair
exists with order $d$ on genus $d-2\geq3$.
Our simplest indistinguishable nontrivial Hempel pair
exists with prime order $p$ on genus $(p-1)/2\geq2$.

The technical heart of this paper is the following calculation.

\begin{theorem}\label{main_sfs} 
	Let $a\geq3$ be an integer and $b_1,\cdots,b_n$ be integers coprime to $a$.
	Let $M$ be a Seifert fiber space 
	with orientable orbifold base and orientable Seifert fibration,
	and with symbol $(g;(a,b_1),\cdots,(a,b_n))$.
	Suppose $g\geq0$, and $a>n\geq 0$, and $b_1+\cdots+b_n=0$.
	\begin{enumerate}
	\item
	If there exist some integer $b^*$ coprime to $a$ and $\nu_1,\cdots,\nu_n\in\{\pm1\}$, 
	such that $b^*b_j\equiv \nu_j\bmod a$	holds for all $j\in\{1,\cdots,n\}$, 
	then for any $s$ coprime to $a$,
	$$\mathrm{TV}_{a,s}(M)
	=\frac{a^{n+2g-2}}{2^{2n+2g-4}}\cdot\frac{1}{{\sin}^{2n+4g-4}(\pi b^*s/a)},$$
	and moreover, if $a$ is odd and $s$ is even,
	$$\mathrm{TV}'_{a,s}(M)
	=\frac{a^{n+2g-2}}{2^{2n+4g-4}}\cdot\frac{1}{{\sin}^{2n+4g-4}(\pi b^*s/a)}.$$
	\item
	Otherwise, for any integer $r\geq3$ divisible by $a$
	and any integer $s$ coprime to $r$,
	$$\mathrm{TV}_{r,s}(M)=0,$$
	and moreover, if $r$ is odd and $s$ is even,
	$$\mathrm{TV}'_{r,s}(M)=0.$$
	\end{enumerate}
\end{theorem}

See Section \ref{Sec-periodic_mapping_tori} 
for review on Seifert fiber spaces and 
the standard notation for their symbols.

Theorem \ref{main_sfs} is proved in Section \ref{Sec-calculation},
by applying a general formula
for calculating the Witten--Reshetikhin--Turaev invariant $\tau_r$ 
of oriented closed Seifert fiber spaces.
The exact formula we invoke is due to Hansen \cite{Hansen_sfs},
while similar calculation in special cases or with likewise strategy
also appear in many other places,
for instance, see \cite{Andersen_rt,LR_sfs,Rozansky_sfs,Taniguchi_sfs}.

Theorem \ref{main_sfs} is formulated by first testing samples of
Seifert fiber spaces (on computer), 
and then observing interesting phenomena.
Luckily, we find the assumptions as in Theorem \ref{main_sfs},
which greatly simplify the situation and the answer.

In the appendix, we prove a splitting formula 
$$\mathrm{TV}_{r,s}(M)=\mathrm{TV}_{3,1}(M)\cdot\mathrm{TV}'_{r,r-s}(M)$$
for $r$ odd and $s$ odd.
This complements a former formula 
$$\mathrm{TV}_{r,s}(M)=\mathrm{TV}_{3,2}(M)\cdot\mathrm{TV}'_{r,s}(M)$$
for $r$ odd and $s$ even,
proved by Detcherry, Kalfagianni, and Yang \cite[Theorem 2.9]{DKY}.
Our proof in the appendix automatically includes the non-orientable case,
although the orientable case suffices for our application.

%

This paper is organized as follows.
In Section \ref{Sec-periodic_mapping_tori}, 
we recall the prelimiary description of periodic mapping tori
as Seifert fiber spaces with vanishing rational Euler number.
In Section \ref{Sec-tv},
we review Turaev--Viro invariants.
In Section \ref{Sec-calculation}, we prove Theorem \ref{main_sfs}.
In Section \ref{Sec-examples}, we prove Theorem \ref{main_hempel}.
In Section \ref{Sec-splitting}, 
we give an elementary proof of 
the aforementioned splitting formula regarding 
$\mathrm{SU}(2)$ Turaev--Viro invariants at odd $r$ and odd $s$.

\subsection*{Acknowledgment} 
The author would like to thank Yuxuan Yang for valuable comments
on a preliminary version of this paper.

\section{Periodic mapping tori}\label{Sec-periodic_mapping_tori}
Let $S$ be a connected orientable closed surface.
The mapping class group $\mathrm{Mod}(S)$ 
consists of all the isotopy classes of orientation-preserving self-homeomorphism of $S$.
For any mapping class $[f]\in\mathrm{Mod}(S)$,
we denote by $M_f$ the mapping torus
$$M_f=\frac{S\times\Real}{(f(x),r)\sim (x,r+1)},$$
which naturally fibers over the oriented circle $\Real/\Integral$
with fiber type $S$ and (backward) monodromy type $[f]$.
The mapping torus $M_f$ is a connected orientable closed $3$--manifold,
whose homeomorphism type depends only on $[f]$.

A \emph{periodic mapping class} refers to a mapping class of finite order.
In this case, 
the mapping torus is a Seifert fiber space.
The Seifert fibration is orientable over an orientable base orbifold,
and has vanishing rational Euler number.
Moreover, 
the Seifert fibration on any periodic mapping torus is unique up to isotopy.
Conversely, any closed Seifert fiber space
with orientable orbifold base and orientable Seifert fibration 
of vanishing rational Euler number arises as 
the mapping torus of some periodic mapping class.
Moreover, 
the genus of surface and 
the conjugacy class of the periodic mapping class
up to inverse are both unique.
This way, periodic mapping classes can be described
equivalently by indicating 
the symbol of its mapping torus, as a Seifert fiber space.
(See \cite[Chapter 1]{AFW_book_group}).

The most general symbol describes 
any (connected, compact) Seifert fiber space with all features,
allowing possibly 
non-orientable orbifold base, non-orientable Seifert fibration, and non-empty boundary.
For dicussing periodic mapping classes,
we only need to consider closed Seifert fiber spaces
with orientable orbifold base and orientable Seifert fibration,
whose (possibly non-normalized) symbol is denoted as
\begin{equation}
(g;(a_1,b_1),\cdots,(a_n,b_n))
\end{equation}
where $g\geq0$ is an integer, and 
where $a_j\geq1$ is an integer and $b_j$ is an integer coprime to $a_j$,
for all $j\in\{1,\cdots,n\}$.
This symbol presents 
a Seifert fiber space (with standard orientation) constructed as follows.

Take a product $3$--manifold
$\Sigma\times S^1$
of a connected closed oriented surface $\Sigma$ of genus $g$
and an oriented circle $S^1$;
take $n$ disjoint embedded disks $D_1,\cdots,D_n\subset\Sigma$;
remove the solid tori $D_j\times S^1$ from $\Sigma\times S^1$,
and refill with solid tori in other ways,
such that the slopes $a_j[\partial D_j]+b_j[S^1]$ 
on $\partial D_j\times S^1$ bound disks 
in the new solid tori.

The base of this Seifert fiber space is a connected, closed, oriented
$2$--orbifold of genus $g$ with $n$ cone points of order $a_1,\cdots,a_n$,
(ordinary and negligible if $a_j=1$).
Its orbifold Euler characteristic equals
$2-2g-\sum_j(1-1/a_j)$.
The rational Euler number of the Seifert fibration
equals $-\sum_j b_j/a_j$,
where the minus sign comes from our convention
on the orientation of $\partial D_j$, 
(as induced by the orientation of $D_j$, rather than $S\setminus (D_1\cup\cdots\cup D_n)$).

The following operations of the symbol 
do not change the homeomorphism type of the resulting Seifert fiber space:
re-ordering all $(a_j,b_j)$; 
inserting or deleting a term $(1,0)$;
replacing one $(a_j,b_j)$ with $(a_j,b_j+a_j)$ 
and another $(a_{j'},b_{j'})$ with $(a_{j'},b_{j'}-a_{j'})$ simultaneously;
or replacing all $(a_j,b_j)$ with $(a_j,-b_j)$ simultaneously.
Note that only the last operation 
changes the orientation-preserving homeomorphism type
of the resulting Seifert fiber space.

For homeomorphic periodic mapping tori,
their symbols are all related by finitely many steps of the above operations.
This is a special case of the classification of Seifert fiber spaces
(see \cite[Theorem 1.10]{Orlik_book_sfs} or \cite[Chapter 12]{Hempel_book}).

The following Lemmas \ref{mapping_torus_sfs} and \ref{mapping_torus_iterate}
actually appear in \cite{Hempel_quotient} in equivalent forms.
We include quick (and slightly different) proofs here for the reader's reference.

\begin{lemma}\label{mapping_torus_sfs}
Let $M$ be a closed Seifert fiber space with orientable orbifold base and orientable Seifert fibration,
and with symbol $(g;(a_1,b_1),\cdots,(a_n,b_n))$.
If the Seifert fibration has vanishing rational Euler characteristic,
then $M$ is homeomorphic to the mapping torus $M_f$
of a periodic mapping class $[f]\in\mathrm{Mod}(S)$.
Moreover,
$[f]$ has order $d=\mathrm{lcm}(a_1,\cdots,a_n)$,
and $S$ has genus $1+(g-1)d+\sum_j(1-1/a_j)d/2$.
\end{lemma}

\begin{proof}
The orbifold fundamental group of the base $\mathcal{O}$ has a presentation
with generators $x_1,y_1,\cdots,x_g,y_g,s_1,\cdots,s_n$
and relatorions $[x_1,y_1]\cdots[x_g,y_g]=s_1\cdots s_n$,
and $s_1^{a_1}=\cdots=s_n^{a_n}=1$.
Under the assumption of vanishing rational Euler number,
The assignments $x_i\mapsto 0\bmod d$, $y_i\mapsto 0\bmod d$,
and $s_j\mapsto -b_jd/a_j\bmod d$ yield a well-defined, surjective homomorphism
$\pi_1(\mathcal{O})\to\Integral/d\Integral$, where $d=\mathrm{lcm}(a_1,\cdots,a_n)$.
The kernel of the homorphism corresponds to a cyclic orbifold covering
$S\to \mathcal{O}$ of degree $d$, where $S$ has no singular cone points.
The generator $1\in\Integral/d\Integral$
corresponds to a deck transformation, 
representing a periodic mapping class $[f]\in\mathrm{Mod}(S)$ of order $d$.
It is elementary to check that $M_f$ is homeomorphic to $M$. 
The Euler characteristic of the surface $S$ 
is equal to $d$ times the rational Euler number of $\mathcal{O}$,
which implies the asserted genus of $S$.
\end{proof}

\begin{lemma}\label{mapping_torus_iterate}
Let $S$ be a connected, close, orientable surface,
and $[f]\in\mathrm{Mod}(S)$ be a periodic mapping class of order $d$.
If the mapping torus $M_f$ has symbol $(g;(a_1,b_1),\cdots,(a_n,b_n))$,
then for any integer $k$ coprime to $d$,
the mapping torus $M_{f^k}$ of the iterate $[f^k]\in\mathrm{Mod}(S)$
has symbol $(g;(a_1,b_1k^*),\cdots,(a_n,b_nk^*))$,
where $k^*$ is any integer satisfying $kk^*\equiv1\bmod d$.
\end{lemma}

\begin{proof}
The mapping torus $M_{f^k}$ naturally cyclically cover $M_f$ of degree $k$.
Since $f$ has order $k$, 
we also see that $M_f$ cyclically covers $M_{f^k}$ of degree $k^*$.
Consider the covering $M_f\to M_{f^k}$.
The pull-back of the Seifert fibration of $M_{f^k}$ is a Seifert fibration on $M_f$.
If $M_f$ has symbol $(g;(a_1,b_1),\cdots,(a_n,b_n))$, by definition,
the preimage of $(\Sigma\setminus (D_1\cup\cdots\cup D_n))\times S^1\subset M_{f^k}$
is also a product $(\Sigma\setminus (D_1\cup\cdots\cup D_n))\times S^1\subset M_{f}$, 
while ordinary fibers (that is, $*\times S^1$) in $M_f$) 
cyclically covers the ordinary fibers in $M_{f^k}$ with degree $k^*$.
Since the slope $a_j[D_j]+b_j[S^1]$ on $\partial D_j\times S^1$ 
is homotopically trivial in $M_f$,
the slope $a_j[D_j]+b_jk^*[S^1]$ on $\partial D_j\times S^1$ is 
homotopically trivial in $M_k$. 
Therefore, $(g;(a_1,b_1k^*),\cdots,(a_n,b_nk^*))$ is a symbol of $M_{f^k}$.
\end{proof}

\section{Turaev--Viro invariants}\label{Sec-tv}
Turaev--Viro invariants are topological invariants of 
closed $3$--manifolds,
arising from representation theory
of quantum groups at roots of unity.
Throughout this paper, we only discuss
Turaev--Viro invariants 
pertaining to the most basic quantum group $\mathcal{U}_q(\mathfrak{sl}_2)$.
In this setting,
there are essentially two series,
namely, the $\mathrm{SU}(2)$ Turaev--Viro invariants 
$\mathrm{TV}_r$ for all integers $r\geq3$,
and the $\mathrm{SO}(3)$ Turaev--Viro invariants 
$\mathrm{TV}'_r$ for all odd integers $r\geq3$.

Throughout this paper, a \emph{closed} manifold
only means a compact manifold with empty boundary, 
possibly disconnected, and possibly non-orientable.

\subsection{Abstract versions}

Let $q^{1/2}$ be a square root of a root of unity $q\neq\pm1$.
Denote by $r\geq3$ the order of $q$.
Note that $q^{1/2}$ must have order $2r$ when $r$ is even,
but may have order $2r$ or $r$ when $r$ is odd.
We denote by $\Rational(q^{1/2})$
the abstract cyclotomic field generated by $q^{1/2}$.

For any 
closed $3$-manifold $M$,
the Turaev--Viro invariant $\mathrm{TV}(M;q^{1/2})$
can be defined at $q^{1/2}$;
if $r$ is odd, the refined Turaev--Viro invariant $\mathrm{TV}'(M;q^{1/2})$
can be defined at $q^{1/2}$ of order $r$.
Both $\mathrm{TV}(M;q^{1/2})$ and $\mathrm{TV}'(M;q^{1/2})$
are topological invariants of $M$,
with totally real values in 
$\Rational(q^{1/2})$.
We call
$\mathrm{TV}$ and $\mathrm{TV}'$ 
the \emph{abstract} $\mathrm{SU}(2)$ and 
the \emph{abstract} $\mathrm{SO}(3)$ Turaev--Viro invariants at $q^{1/2}$
(or at level $r-2$), respectively.

Most details about the actual construction of these invariants 
are not be necessary in the sequel, except the appendix.
We record them below for the sake of clarity.

Let $\mathscr{T}=(V,E,F,T)$ be any finite simplicial $3$--complex,
where the items denote the sets of vertices, edges, faces, and tetrahedra, respectively.
Denote $I_r=\{0,1,\cdots,r-2\}$.
A triple $(i,j,k)\in I_r\times I_r\times I_r$ is said to be \emph{admissible}
if the numbers $(i+j+k)/2$, $(i+j-k)/2$, $(j+k-i)/2$, $(k+i-j)/2$
all stay in $I_r$.
An \emph{admissible coloring} of $(M,\mathscr{T})$ (of \emph{level} $r-2$)
refers to a map $c\colon E\to I_r$,
such that on any face, the edge colors (that is, values of $c)$
form an admissible triple.
Denote by $\mathcal{A}_r=\mathcal{A}_r(M,\mathscr{T})$
the set of all admissible colorings of level $r-2$.
If $r$ is odd,
denote by $I'_r=\{0,2,\cdots,r-3\}$
the subset of even elements of $I_r$.
Denote by $\mathcal{A}'_r=\mathcal{A}'_r(M,\mathscr{T})$
the subset of $\mathcal{A}_r$
consisting of admissible colorings with values in $I'_r$.

Let $M$ be a (possibly disconnected, possibly non-orientable)
closed $3$--manifold.
Take a triangulation $\mathscr{T}=(V,E,F,T)$ of $M$.
Then invariants $\mathrm{TV}$ and $\mathrm{TV}'$ of $M$ at $q^{1/2}$
can be expressed in terms of $\mathscr{T}$ as follows:
\begin{equation}\label{tv_def}
\mathrm{TV}\left(M;q^{1/2}\right)=\left(\frac{(q^{1/2}-q^{-1/2})^2}{-2r}\right)^{|V|}\cdot\sum_{c\in\mathcal{A}_r}
\left|\mathscr{T}\right|_c;
\end{equation}
if $r$ is odd and $q^{1/2}$ has order $r$, 
\begin{equation}\label{tv_refined_def}
\mathrm{TV}'\left(M;q^{1/2}\right)=\left(\frac{(q^{1/2}-q^{-1/2})^2}{-r}\right)^{|V|}\cdot\sum_{c\in\mathcal{A}'_r}
\left|\mathscr{T}\right|_c,
\end{equation}
where $|V|$ denotes the number of vertices in $\mathscr{T}$,
and where $|\mathscr{T}|_c$ is explained in Notation \ref{color_terms} below.
Note that our notations follow the reformulation as in \cite[Appendix A]{DKY},
where the factors in $|e|_c$, $|f|_c$, and $|t|_c$ are grouped slightly differently
than those in \cite{TV_invariant}, so as to avoid unnecessary square roots.

As it turns out, the values of the above expressions in $\Rational(q^{1/2})$
are independent of the auxiliary choice of the triangulation $\mathscr{T}$  \cite{TV_invariant}. 
Therefore, 
$\mathrm{TV}(M;q^{1/2})$ and $\mathrm{TV}'(M;q^{1/2})$ are indeed topological invariants of $M$.

\begin{notation}\label{color_terms}\
\begin{enumerate}
\item
Denote
$$[n]!=
\begin{cases}
[1]\cdot[2]\cdot\cdots\cdot[n]
&n=1,2,\cdots,r-1\\
1 & n=0
\end{cases}$$
where
$$[n]=\frac{q^{n/2}-q^{-n/2}}{q^{1/2}-q^{-1/2}}.$$
Note that the quantum integers
$[1],[2],\cdots,[r-1]$ take totally real nonzero values in $\Rational(q^{1/2})$,
as $q$ is a primitive $r$-th root of unity ($r\geq3$).
\item
For any $e\in E$ and $c\in\mathcal{A}_r$,
$$|e|_c =(-1)^{i}\cdot[i+1]$$
where $i$ is the color of $e$ under $c$.
\item
For any $f\in F$ and $c\in\mathcal{A}_r$,
$$|f|_c=(-1)^{-S}\cdot\frac{[S-i]!\cdot[S-j]!\cdot[S-k]!}{[S+1]!}$$
where $i,j,k$ are the edge colors of $f$ under $c$, and where $S=(i+j+k)/2$.
\item
For any $t\in T$ and $c\in\mathcal{A}_r$, denote
$$|t|_c
=\sum_{z}\frac{(-1)^z\,[z+1]!}{\prod_a[z-T_a]!\cdot\prod_b[Q_b-z]!}$$
where $(i,j,k),(i,m,n),(j,l,n),(k,l,m)$ 
are the edge colors of the faces of $t$ under $c$,
and where $T_1=(i+j+k)/2$, $T_2=(i+m+n)/2$, $T_3=(j+l+n)/2$, $T_4=(k+l+m)/2$,
$Q_1=(i+j+l+m)/2$, $Q_2=(i+k+l+n)/2$, $Q_3=(j+k+m+n)/2$.
The index $a$ ranges in $\{1,2,3,4\}$, and $b$ in $\{1,2,3\}$,
and $z$ from $\max_a T_a$ to $\min_b Q_b$.
\item
For any $c\in\mathcal{A}_r$, denote
$$|\mathscr{T}|_c=\prod_{e\in E} |e|_c\cdot\prod_{f\in F} |f|_c\cdot \prod_{t\in T} |t|_c.$$
\end{enumerate}
\end{notation}

\subsection{Specialized versions}
In the literature, \emph{the} $\mathrm{SU}(2)$ and \emph{the} $\mathrm{SO}(3)$ Turaev--Viro invariants
often refer to specialization of the abstract versions
at customary complex roots of unity.
To be precise, these refer to the numerical quantities
$\mathrm{TV}_r$ and $\mathrm{TV}'_r$ below.

\begin{notation}\label{tv_customary}
Let $r\geq3$ be an integer and $s$ be an integer coprime to $r$.
The following expressions are all evaluated by specializing $\Rational(q^{1/2})\to\Complex$.
\begin{enumerate}
\item Denote
$$\mathrm{TV}_{r,s}(M)=\mathrm{TV}\left(M;q^{1/2}=e^{\sqrt{-1}\cdot\pi s/r}\right),$$
\item If $r$ is odd and $s$ is even, denote
$$\mathrm{TV}'_{r,s}(M)=\mathrm{TV}'\left(M;q^{1/2}=e^{\sqrt{-1}\cdot\pi s/r}\right).$$
\item In informal discussion,
we often write $\mathrm{TV}_r$ and $\mathrm{TV}'_r$, assuming $s$ implicitly fixed.
\end{enumerate}
\end{notation}

\begin{lemma}\label{tv_property}
Let $M,N$ be any closed $3$--manifolds.
\begin{enumerate}
\item $\mathrm{TV}_{r,s}(M)\in\Real$.
\item
$\mathrm{TV}_{r,s}(M\sqcup N)=\mathrm{TV}_{r,s}(M)\cdot\mathrm{TV}_{r,s}(N)$.
\item $\mathrm{TV}_{r,s}(S^2\times S^1)=1.$
\end{enumerate}
If $r$ is odd and $s$ is even, the same statements hold with $\mathrm{TV}'_{r,s}$ in place of $\mathrm{TV}_{r,s}$.
\end{lemma}

The first statement follow immediately
from the observation that $\mathrm{TV}$ and $\mathrm{TV}'$ 
can be written as rational functions over $\Rational$
in $[1],\cdots,[r-1]$ and $(q^{1/2}-q^{-1/2})^2$;
see (\ref{tv_def}), (\ref{tv_refined_def}), and Notation \ref{color_terms}.
Note $[n]=\sin(\pi sn/r)/\sin(\pi s/r)$ and $(q^{1/2}-q^{-1/2})^2=-4\sin^2(\pi s/r)$
evaluated at $q^{1/2}=e^{\sqrt{-1}\cdot\pi s/r}$.
The second statement is also obvious by definition.
The last statement appears in Turaev and Viro's original paper \cite[Section 8.1.B]{TV_invariant}.

\begin{lemma}\label{tv_r_odd}
Let $M$ be any closed $3$--manifold.
\begin{enumerate}
\item $\mathrm{TV}_{r,s}(M)=\mathrm{TV}_{r,-s}(M)=\mathrm{TV}_{r,s+2r}(M)$.
If $r$ is odd and $s$ is even, the same identities hold with $\mathrm{TV}'$ in place of $\mathrm{TV}$.
\item
If $s$ is odd, $\mathrm{TV}_{r,s}(M)$ is Galois conjugate to $\mathrm{TV}_{r,1}(M)$;
if $s$ is even (hence $r$ odd),
$\mathrm{TV}_{r,s}(M)$ is Galois conjugate to $\mathrm{TV}_{r,r-1}(M)$.
If $r$ is odd and $s$ is even, the same statements hold with $\mathrm{TV}'_{r,s}$ in place of $\mathrm{TV}_{r,s}$.
\item If $r$ is odd,
$$\mathrm{TV}_{r,s}(M)=
\begin{cases}
\mathrm{TV}_{3,2}(M)\cdot\mathrm{TV}'_{r,s}(M)&s\mbox{ even}\\
\mathrm{TV}_{3,1}(M)\cdot\mathrm{TV}'_{r,r-s}(M)&s\mbox{ odd}
\end{cases}
$$
\item
Denote by $\beta_i$ the dimension of $H_i(M;\Integral/2\Integral)$
over $\Integral/2\Integral$, and by $w_1\in H^1(M;\Integral/2\Integral)$ 
the first Stifel--Whitney class of $M$.
$$\mathrm{TV}_{3,2}(M)=2^{-\beta_0(M)+\beta_2(M)},$$
and
$$\mathrm{TV}_{3,1}(M)=2^{-\beta_0(M)}\cdot\sum_{t}(-1)^{\langle t^3+w_1^2t,[M]\rangle}$$
where the index $t$ ranges over $H^1(M;\Integral/2\Integral)$.
\end{enumerate}
\end{lemma}

The first and the second statements are again obvious properties
of the defining expressions.

The third statement
is proved when $s$ is even 
by Detcherry, Kalfagianni, and Yang \cite[Theorem 2.9]{DKY},
and can be easily derived from well-known facts when $s$ is odd,
assuming $M$ orientable, (see Lemmas \ref{rt_r_odd} and \ref{tv_and_rt}).
We supply a proof without assuming orientability
in the appendix Section \ref{Sec-splitting}.

The formulas for $\mathrm{TV}_{3,s}$ 
in the fourth statement are due to Turaev and Viro 
\cite[Section 9.3.A]{TV_invariant}.

\subsection{Relation to Witten--Reshetikhin--Turaev invariants}
For connected oriented 
closed $3$--manifolds,
the Witten--Reshetikhin--Turaev invariants 
are invariant under orientation-preserving homeomorphisms.
These invariants were suggested by Witten \cite{Witten_qft},
and the first mathematically rigorous construction
was due to Reshetikhin and Turaev \cite{RT_invariant}.

Following Kirby and Melvin \cite{KM_invariant},
we denote by $\tau_r$ 
the $\mathrm{SU}(2)$ Witten--Reshetikhin--Turaev invariants,
defined for any integer $r\geq3$, and 
by $\tau'_r$ 
the $\mathrm{SO}(3)$ Witten--Reshetikhin--Turaev invariants, 
defined for any odd integer $r\geq3$.
Note that $\tau_r$ is slight modification of the original construction by
Reshetikhin and Turaev \cite{RT_invariant},
differing by a factor of absolute value $1$ (depending on the first Betti number),
and $\tau'_r$ is introduced by Kirby and Melvin \cite[Section 8]{KM_invariant}.
The invariant $\tau_r$ corresponds to
the $4r$--th primitive complex root of unity
$q^{1/4}=e^{\sqrt{-1}\cdot2\pi/4r}$ by construction,
whereas $\tau'_r$ 
is obtained by modifying the defining expression of $\tau_r$
when $r$ is odd.
Upon suitable interpretation, 
$\tau'_r$ appears to correspond to
the $2r$--th primitive complex root of unity $q^{1/4}=e^{\sqrt{-1}\cdot\pi(r-1)/4r}$,
(compare \cite[Theorem B]{BHMV_invariant}).
Both $\tau_r$ and $\tau'_r$ take values in $\Complex$.

The following properties characterize the normalization of $\tau_r$ and $\tau'_r$.

\begin{lemma}\label{rt_property}
Let $M,N$ be any connected oriented closed $3$--manifolds.
\begin{enumerate}
\item $\tau_r(M)=\overline{\tau_r(-M)}$.
\item
$\tau_r(M\#N)=\tau_r(M)\cdot\tau_r(N)$.
\item 
$\tau_r(S^2\times S^1)=\sqrt{r/2}\,/\sin(\pi/r)$.
Hence, $\tau_r(S^3)=1$.
\end{enumerate}
If $r$ is odd, the same statements hold with $\tau'_r$ in place of $\tau_r$,
and $\sqrt{r/4}$ in place with $\sqrt{r/2}$.
\end{lemma}

See \cite[Section 1 and (5.11)]{KM_invariant} and \cite[Proposition 6.10]{BHMV_invariant}.

\begin{lemma}\label{rt_r_odd}
Let $M$ be any connected oriented closed $3$--manifold.
If $r$ is odd,
$$\tau_r(M)=
\begin{cases}
\tau_3(M)\cdot\tau'_r(M)&r\equiv3\bmod4\\
\overline{\tau_3(M)}\cdot\tau'_r(M)&r\equiv1\bmod4
\end{cases}
$$
\end{lemma}

See \cite[Corollary 8.9 and Theorem 8.10]{KM_invariant}.
See also \cite[Theorem 6.11]{KM_invariant} for characterization of $\tau_3(M)$
in terms of classical topological invariants.

The Turaev--Viro invariants $\mathrm{TV}_r$ and $\mathrm{TV}'_r$
are essentially the absolute value squares
of Witten--Reshetikhin--Turaev invariants,
or precisely as follows.

\begin{lemma}\label{tv_and_rt}
Let $M$ be a connected closed $3$--manifold.
Let $r\geq3$ be any integer.
\begin{enumerate}
\item
If $M$ is oriented,
$$\frac{\mathrm{TV}_{r,1}(M)}{\mathrm{TV}_{r,1}(S^3)}=|\tau_r(M)|^2,$$
and if $r$ is odd,
$$\frac{\mathrm{TV}'_{r,r-1}(M)}{\mathrm{TV}'_{r,r-1}(S^3)}=|\tau'_r(M)|^2.$$	
\item
If $M$ is non-orientable, 
$$\frac{\mathrm{TV}_{r,1}(M)}{\mathrm{TV}_{r,1}(S^3)}=\tau_r(W),$$
and if $r$ is odd,
$$\frac{\mathrm{TV}'_{r,r-1}(M)}{\mathrm{TV}'_{r,r-1}(S^3)}=\tau'_r(W),$$
where $W$ denotes the orientable connected double cover of $M$.
Note that $W$ can be canonically constructed,
with a canonical orientation 
and an orientation-reserving deck transformation.
\item
$$\mathrm{TV}_{r,1}(S^3)=\frac{2}{r}\cdot\sin^2(\pi/r),$$
and if $r$ is odd,
$$\mathrm{TV}'_{r,r-1}(S^3)=\frac{4}{r}\cdot\sin^2(\pi/r).$$
\end{enumerate}
\end{lemma}

See \cite[Theorems 4.1.1 and 4.4.1]{Turaev_book_quantum}
for systematic proofs of the formulas regarding $\mathrm{TV}_r$;
see also \cite{Roberts_tv} for an elegant proof
for the oriented case regarding $\mathrm{TV}_r$.
The formulas regarding $\mathrm{TV}'_r$ 
can be easily derived from the $\mathrm{TV}_r$ formulas
using Lemmas \ref{rt_r_odd} and \ref{tv_r_odd}.
The values for $\mathrm{TV}_{r,1}(S^3)$ and $\mathrm{TV}'_{r,r-1}(S^3)$
can be obtained immediately by taking $M=S^2\times S^1$
and applying Lemmas \ref{tv_property} and \ref{rt_r_odd}.

\subsection{Perspective of TQFT}
The Turaev--Viro invariants 
fit into the general framework of
$(2+1)$--dimensional topological quantum field theories (TQFT) \cite{Atiyah_tqft}.
We describe below focusing on the $\mathrm{SU}(2)$ Turaev--Viro invariants.
The discussion regarding $\mathrm{SO}(3)$ Turaev--Viro invariants is completely similar.

Turaev and Viro construct 
a functor $\mathcal{Z}^{\mathrm{TV}}$ from the $(2+1)$--dimensional cobordism category
to the category of Hermitian vector spaces (and linear homomorphisms) 
over the abstract cyclotomic field $\mathbb{K}=\Rational(q^{1/2})$,
with respect to the involution 
$*$ on $\mathbb{K}$
as provided by the Galois transformation $q^{1/2}\mapsto q^{-1/2}$.
To any oriented closed surface $S$,
there is a finite-dimensional vector space $\mathcal{Z}^{\mathrm{TV}}(S)$,
equipped with a nondegenerate Hermitian pairing
$$\langle\_,\_\rangle\colon \mathcal{Z}^{\mathrm{TV}}(S)\times \mathcal{Z}^{\mathrm{TV}}(S)\to \mathbb{K}.$$
(Being \emph{Hermitian} means $\mathbb{K}$--sesquilinear and $*$--symmetric.)
To any cobordism $M$ from $S_0$ to $S_1$
(that is, an oriented compact $3$-manifold $M$ with bipartite boundary
$\partial M=\partial_- M\sqcup \partial_+ M=(-S_0)\sqcup S_1$,
up to boundary-fixing homomeorphisms),
there is a linear homomorphism 
$$\mathcal{Z}^{\mathrm{TV}}(M)\colon \mathcal{Z}^{\mathrm{TV}}(S_0)\to \mathcal{Z}^{\mathrm{TV}}(S_1).$$
The assignment $\mathcal{Z}^{\mathrm{TV}}$ is functorial,
and satisfies Atiyah's Hermitian TQFT axioms \cite[Section 2]{Atiyah_tqft}:
$\mathcal{Z}^{\mathrm{TV}}(-S)=\mathcal{Z}^{\mathrm{TV}}(S)^*$;
$\mathcal{Z}^{\mathrm{TV}}(-M)=\mathcal{Z}^{\mathrm{TV}}(M)^*$;
$\mathcal{Z}^{\mathrm{TV}}(S'\sqcup S'')=\mathcal{Z}^{\mathrm{TV}}(S')\otimes_{\mathbb{K}}\mathcal{Z}^{\mathrm{TV}}(S'')$;
and $\mathcal{Z}^{\mathrm{TV}}(\emptyset)=\mathbb{K}$,
($\emptyset$ denoting the empty surface, having a unique orientation by convention).

Turaev and Viro show that $\mathrm{TV}$ comes from a functor $\mathcal{Z}^{\mathrm{TV}}$ as above,
in the sense that the identity
$$\mathrm{TV}\left(M;q^{1/2}\right)=\mathcal{Z}^{\mathrm{TV}}(M)$$
holds for any closed $3$-manifold $M$.
Here, $M$ is treated as a cobordism between empty surfaces,
and $\mathcal{Z}^{\mathrm{TV}}(\emptyset,\emptyset)=\mathrm{End}_{\mathbb{K}}(\mathbb{K})$
is identified as $\mathbb{K}$. See \cite[Section 2.3]{TV_invariant}.

Being a TQFT functor, $\mathcal{Z}^{\mathrm{TV}}$ naturally induces
a $\mathbb{K}$--linear representation 
$$\mathrm{Mod}(S)\to\mathrm{GL}\left(\mathcal{Z}^{\mathrm{TV}}(S)\right)$$
of the mapping class group $\mathrm{Mod}(S)$
for any oriented closed surface $S$ \cite[Section 2.4]{TV_invariant}.
Specializing $q^{1/2}$ to complex roots of unity as in Notation \ref{tv_customary},
we obtain a complex linear representation, denoted as
\begin{equation}\label{rep_r_s}
\rho^{\mathrm{TV}}_{r,s}\colon \mathrm{Mod}(S)\to\mathrm{GL}\left(\mathcal{Z}^{\mathrm{TV}}_{r,s}(S)\right)
\end{equation}
for each integer $r\geq3$ and any integer $s$ coprime to $r$.
These representations preserve the specialized Hermitian pairings $\langle\_,\_\rangle_{r,s}$,
whose signatures depend on both $r$ and $s$, but are not necessarily (Hilbert) unitary.
We refer to these as the $\mathrm{SU}(2)$ \emph{Turaev--Viro TQFT representations} (at level $r-2$)
of $\mathrm{Mod}(S)$.

The following formula is useful implication of TQFT axioms \cite[Section 2]{Atiyah_tqft}.

\begin{lemma}\label{mapping_torus_tv}
Let $r\geq3$ be any integer and $s$ be any integer coprime to $r$.
For any oriented closed surface $\Sigma$,
and any mapping class $[f]\in\mathrm{Mod}(S)$,
$$\mathrm{TV}_{r,s}(M_f)=\mathrm{tr}_\Complex\left(\rho^{\mathrm{TV}}_{r,s}([f])\right),$$
where $M_f$ denotes the mapping torus of $f$.
\end{lemma}

If $S$ is connected,
the complex dimension of the representation $\rho^{\mathrm{TV}}_{r,s}$
depends only on $r$ and genus $g$ of $S$.
This fact can be derived from Lemmas \ref{mapping_torus_tv} and \ref{tv_r_odd}
by considering the mapping torus of $f=\mathrm{id}_S$.
Verlinde type formulas for these dimensions 
can be derived from Lemma \ref{tv_and_rt} and known formulas 
about Witten--Reshetikhin--Turaev invariants \cite[Corollary 1.16]{BHMV_tqft}.
However, Witten--Reshetikhin--Turaev invariants 
only come from generalized TQFTs,
which require extra structures for resolving framing anomaly
\cite{BHMV_tqft}.
That is why they only naturally lead to 
projective linear representations of surface mapping class groups. 

\section{Calculation}\label{Sec-calculation}
This section is devoted to the proof of Theorem \ref{main_sfs}.

To restate our task,
we consider a closed Seifert fiber space $M$
with orientable orbifold base and orientable fibration,
and with symbol $(g;(a_1,b_1),\cdots,(a_n,b_n))$,
such that
$$a_1=a_2=\cdots=a_n=a,$$
and 
$$b_1+b_2+\cdots+b_n=0.$$
Moreover, we suppose $a\geq 3$ and $a>n\geq0$.
We compute $\mathrm{TV}_r$ and $\mathrm{TV}'_r$ of $M$
for $r=a$, and once they vanish, 
we show that 
they must also vanish for any $r$ divisible by $a$.

We invoke the following explicit formula for computing
the Witten--Reshetikhin--Turaev invariant $\tau_r$ of Seifert fiber spaces.
Recall $\tau_r(S^1\times S^2)=\sqrt{r/2}/\sin(\pi/r)$ (Lemma \ref{rt_property}).

\begin{lemma}\label{sfs_rt}
	Let $M$ be
	a closed Seifert fiber space 
	with orientable orbifold base and orientable fibration,
	and with symbol $(g;(a_1,b_1),\cdots,(a_n,b_n))$,
	where $g,n\geq0$ are integers, 
	and 
	where $a_j\geq0$ and $b_j$ are coprime pairs of integers for 
	$j=1,\cdots,n$.
	Orient $M$ by orienting the base and the fibers,
	such that the rational Euler number of the Seifert fibration is
	$$E=-\sum_j b_j/a_j.$$
	Then,
	$$\frac{\tau_{r}(M)}{\tau_r(S^2\times S^1)}
	=\frac{r^{g-1}\cdot U_r\cdot Z_r}{2^{n+g-1}\sqrt{\prod_j a_j}}$$
	where
	$$Z_r=\sum_{(\gamma,\boldsymbol{\mu},\boldsymbol{m})}\left\{
	\frac{e^{\sqrt{-1}\cdot\frac{\pi \gamma^2 E}{2r}}}{{\sin}^{n+2g-2}(\pi\gamma/r)}
	\cdot\prod_j
	\mu_j
	e^{\sqrt{-1}\cdot\left(\frac{-\pi(2rm_j+\mu_j)\gamma}{a_jr}+
	\frac{-2\pi (rm_j^2+\mu_jm_j)b^*_j}{a_j}\right)}
	\right\}
	$$
	and
	$$U_r=(-1)^n\cdot
	e^{\sqrt{-1}\cdot\left(\frac{3\pi}{2r}-\frac{3\pi}4\right)\cdot\mathrm{sgn}(E)}
	\cdot e^{\sqrt{-1}\cdot\frac{\pi\left(E+12\cdot\sum_j \mathrm{s}(b_j,a_j)\right)}{2r}}
	$$
	Here, any $j$ ranges over $\{1,\cdots,n\}$,
	and $(\gamma,\boldsymbol{\mu},\boldsymbol{m})=(\gamma,(m_1,\cdots,m_n),(\mu_1,\cdots,\mu_n))$
	ranges over $\{1,2,\cdots,r-1\}\times\{\pm1\}^n\times \Integral/a_1\Integral\times \cdots \times\Integral/a_n\Integral$.
	The notation $b_j^*$ denotes any congruence inverse of $b_j$ modulo $a_j$,
	namely, $b_jb_j^*\equiv1\bmod a_j$;
	$\mathrm{sgn}(E)$ denotes the sign of $E$, with value $\pm1$ or $0$;
	$\mathrm{s}(b_j,a_j)$
	denotes the Dedekind sum $(4a_j)^{-1}\cdot\sum_{l\,\in\{1,2,\cdots,a_j-1\}}\cot(\pi l/a_j)\cot(\pi l b_j/a_j)$.
\end{lemma}

See \cite[Theorem 8.4]{Hansen_sfs} for this formula;
see also Remark \ref{sfs_rt_remark} below for clarification about 
different notations and normalizations.
Hansen actually obtains the most general formula
that applies to any orientable closed Seifert fiber space,
including those with non-orientable orbifold base.
We have only stated here the case with orientable orbifold base.
An equivalent formula for this case
is formerly obtained by Rozansky \cite{Rozansky_sfs}.

\begin{remark}\label{sfs_rt_remark}
Our notation $\tau_r$ agrees with Kirby--Melvin \cite{KM_invariant}, 
differing from Hansen's notation by our factor $1/\tau_r(S^2\times S^1)$.
In \cite[Section 8]{Hansen_sfs}, Hansen's $\tau_r(M)$ 
is defined as $\tau_{(\mathcal{V}_t,\mathcal{D})}(M)$ therein;
as pointed out in \cite[Appendix A]{Hansen_sfs},
Kirby--Melvin's $\tau_r(M)$ is equal to 
$\tau_A(M)=\mathcal{D}\cdot\tau_{(\mathcal{V}_t,\mathcal{D})}(M)$,
where $\mathcal{D}$ is specified as $\sqrt{r/2}/\sin(\pi/r)$ 
in the equation (38) therein.
More directly,
one may check by evaluating the formula in \cite[Theorem 8.4]{Hansen_sfs} 
for $S^3$ (setting $g=0$, $b=1$, and $n=0$) and $S^2\times S^1$ 
(setting $g=0$, $b=0$ and $n=0$) in the simplest case $r=3$.
\end{remark}

\begin{lemma}\label{sfs_rt_simplified}
Under the assumptions of Theorem \ref{main_sfs},
and assuming $r$ divisible by $a$,
the term $Z_r=Z_r(M)$ as in Lemma \ref{sfs_rt} becomes
$$Z_r=
\sum_{(\gamma,\boldsymbol{\mu})}\left\{
\frac{e^{\sqrt{-1}\cdot\frac{-\pi\gamma\sum_j\mu_j}{ar}}\cdot\prod_j \mu_j }{{\sin}^{n+2g-2}(\pi\gamma/r)}
\cdot
\prod_j
\sum_{m_j}e^{\sqrt{-1}\cdot\frac{-2\pi(\gamma +b_j^*\mu_j)m_j}{a}}
\right\}
$$
where $(\gamma,\boldsymbol{\mu},\boldsymbol{m})$ ranges over
$\{1,\cdots,r-1\}\times\{\pm1\}^n\times (\Integral/a\Integral)^n$.
\end{lemma}

\begin{proof}
In the expression of $Z_r$ in Lemma \ref{sfs_rt},
if any $a_j$ divides $r$, 
we can ignore the term $rm_j^2$ in the exponent of the $j$-th factor in the product;
if $E=0$, we can ignore the factor that involves $\gamma^2$ on the exponent.
Therefore, under these conditions, the expression of $Z_r$ can be rearranged into
\begin{eqnarray*}
Z_r(M)&=&\sum_{(\gamma,\boldsymbol{\mu})}\left\{
\frac{\prod_j \mu_j e^{\sqrt{-1}\cdot\frac{-\pi\gamma\mu_j}{a_jr}}}{{\sin}^{n+2g-2}(\pi\gamma/r)}
\cdot\sum_{\boldsymbol{m}}
e^{\sqrt{-1}\cdot\sum_j\left(\frac{-2\pi\gamma m_j}{a_j}+\frac{-2\pi b^*_j\mu_j m_j}{a_j}\right)}
\right\}\\
&=&\sum_{(\gamma,\boldsymbol{\mu})}\left\{
\frac{e^{\sqrt{-1}\cdot\sum_j\frac{-\pi\gamma\mu_j}{a_jr}}\cdot\prod_j \mu_j }{{\sin}^{n+2g-2}(\pi\gamma/r)}
\cdot
\sum_{\boldsymbol{m}}
e^{\sqrt{-1}\cdot\sum_j\frac{-2\pi(\gamma +b_j^*\mu_j)m_j}{a_j}}
\right\} \\
&=&
\sum_{(\gamma,\boldsymbol{\mu})}\left\{
\frac{e^{\sqrt{-1}\cdot\sum_j\frac{-\pi\gamma\mu_j}{a_jr}}\cdot\prod_j \mu_j }{{\sin}^{n+2g-2}(\pi\gamma/r)}
\cdot
\prod_j
\sum_{m_j}e^{\sqrt{-1}\cdot\frac{-2\pi(\gamma +b_j^*\mu_j)m_j}{a_j}}
\right\}
\end{eqnarray*}
where 
$(\gamma,\boldsymbol{\mu})$ ranges in $\{1,2,\cdots,r-1\}\times\{\pm1\}^n$,
and $j$ in $\{1,\cdots,n\}$, and $m_j$ in $\Integral/a_j\Integral$.
In particular, the simplification applies 
as we assume $a_1=a_2=\cdots=a_n=a$, and $b_1+b_2+\cdots+b_n=0$,
and $r$ divisible by $a$.
\end{proof}

\begin{lemma}\label{sfs_vanishing_criterion}
Under the assumptions of Theorem \ref{main_sfs},
and assuming $r$ divisible by $a$,
if there does not exist any $\boldsymbol{\mu}\in \{\pm1\}^n$
that satisfies the congruence equations
$$b_1^*\mu_1\equiv b_2^*\mu_2\equiv \cdots\equiv b_n^*\mu_n\mod a,$$
then 
$$Z_r=0$$
where $Z_r$ is the term as in Lemma \ref{sfs_rt}.
\end{lemma}

\begin{proof}
For any fixed $(\gamma,\boldsymbol{\mu})$,
we observe
$$\sum_{m_j\,\in\Integral/a\Integral}e^{\sqrt{-1}\cdot\frac{-2\pi(\gamma +b_j^*\mu_j)m_j}{a}}
=\begin{cases} a &\mbox{if }\gamma+b_j^*\mu_j\equiv0 \bmod a\\
0 &\mbox{otherwise}\end{cases}$$
Therefore, 
in the simplified expression of $Z_r$ as in Lemma \ref{sfs_rt_simplified},
the summand corresponding to $(\gamma,\boldsymbol{\mu})$ is nonzero
if and only if $\gamma+b^*_j\mu_j\equiv 0\bmod a$ holds for all $j\in\{1,\cdots,n\}$.
For the sum $Z_r$ to be nonzero,
there has to be some $\gamma\in\{1,2,\cdots,r-1\}$ that 
satisfies the above condition
for some $\boldsymbol{\mu}\in\{\pm1\}^n$,
then there has to be some $\boldsymbol{\mu}$
that satisfies
$b_1^*\mu_1\equiv b_2^*\mu_2\equiv \cdots\equiv b_n^*\mu_n\mod a$.
\end{proof}

\begin{lemma}\label{sfs_nonvanishing}
Under the assumptions of Theorem \ref{main_sfs},
if there exists some integer $b^*$ coprime to $a$, and 
some $\boldsymbol{\nu}\in \{\pm1\}^n$, such that
$b^*\equiv b^*_j\nu_j\bmod a$ holds for all $j\in\{1,\cdots,n\}$,
then
$$Z_a=
\frac{2\cdot a^{n}\cdot\prod_j\nu_j}{{\sin}^{n+2g-2}(\pi b^*/a)}$$
where $Z_a$ is the term as in Lemma \ref{sfs_rt} with $r=a$.
\end{lemma}

\begin{proof}
Possibly after permuting $\{1,\cdots,n\}$,
we may assume
$\nu_j=1$ for $j=1,\cdots,m$ and $\nu_j=-1$ for $j=m+1,\cdots,n$.
We may also assume $b^*\in\{1,2,\cdots,a-1\}$
without changing its residue class modulo $a$.

For $r=a\geq3$, there are only two nonzero summands in $Z_r$, 
and their corresponding to $(\gamma,\boldsymbol{\mu})$ are 
$$(\gamma,\boldsymbol{\mu})=(a-b^*,(\underbrace{1,\cdots,1}_{m}, \underbrace{-1,\cdots,-1}_{n-m}))$$
and
$$(\gamma,\boldsymbol{\mu})=(b^*,(\underbrace{-1,\cdots,-1}_{m}, \underbrace{1,\cdots,1}_{n-m})).$$

By our assumption $\sum_j b_j=0$ in Theorem \ref{main_sfs},
$(2m-n)b\equiv mb+(n-m)(-b)\equiv\sum_j b_j=0\bmod a$,
so $n-2m$ is divisible by $a$.
By our assumption $a>n\geq0$ in Theorem \ref{main_sfs}, 
we must have $|n-2m|<a$, and hence $n-2m=0$.
So, we observe
$$(-1)^{n-m}=(-1)^{m},$$
which is useful below.

We compute
\begin{eqnarray*}
Z_a(M)&=&
\frac{(-1)^{n-m}\cdot e^{\sqrt{-1}\cdot\frac{-\pi(2m-n)(a-b^*)}{a^2}}}{{\sin}^{n+2g-2}(\pi (a-b^*)/a)}\cdot a^n+
\frac{(-1)^{m}\cdot e^{\sqrt{-1}\cdot\frac{-\pi(n-2m)b^*}{a^2}}}{{\sin}^{n+2g-2}(\pi b^*/a)}\cdot a^n\\
&=&\frac{2\cdot a^{n}\cdot(-1)^{n-m}}{{\sin}^{n+2g-2}(\pi b^*/a)}\\
&=&
\frac{2\cdot a^{n}\cdot\prod_j\nu_j}{{\sin}^{n+2g-2}(\pi b^*/a)},
\end{eqnarray*}
as desired.
\end{proof}

\begin{lemma}\label{sfs_cohomology}
	Let $M$ be
	a closed Seifert fiber space 
	with orientable orbifold base and orientable fibration,
	and with symbol $(g;(a_1,b_1),\cdots,(a_n,b_n))$,
	where $g,n\geq0$ are integers, 
	and 
	where $a_j\geq1$ and $b_j$ are coprime pairs of integers, for 
	$j=1,\cdots,n$, and satisfy 
	$b_1/a_1+\cdots+b_n/a_n=0$.
	If $a_1,\cdots,a_n$ are all odd,
	$$\mathrm{TV}_{3,1}(M)=\mathrm{TV}_{3,2}(M)=2^{2g}.$$
\end{lemma}

\begin{proof}
Denote 
$$\mathrm{lcm}(a_1,\cdots,a_n)=d.$$ 
If $a_1,\cdots,a_n$ are all odd, $d$ is also odd.

The fundamental group of $M$ has a presentation
with generators 
$$x_1,y_1,\cdots,x_g,y_g,s_1,\cdots,s_n,f$$
and relations
$$
\begin{cases}
s_1\cdots s_n=[x_1,y_1]\cdots[x_g,y_g]\\
x_if=fx_i & i=1,\cdots,g\\
y_if=fy_i & i=1,\cdots,g\\
s_jf=fs_j & j=1,\cdots,n\\
s_j^{a_j}f^{b_j}=1 & j=1,\cdots,n\\
\end{cases}
$$
With $\Integral/2\Integral$ coefficients,
we can eliminate any $[s_j]$
using the relation $a_j [s_j] + b_j [f]=0$,
since $a_j$ is odd.
Then the first relation is equivalent to $-(b_1d/a_1+\cdots+b_nd/a_n)[f]=0$ 
over $\Integral/2\Integral$, having no effect by our assumption $b_1/a_1+\cdots+b_n/a_n=0$.
It follows that $H_1(M;\Integral/2\Integral)$ is 
freely generated by $[x_1],[y_1],\cdots,[x_g],[y_g],[f]$ over $\Integral/2\Integral$.
Then the $\Integral/2\Integral$ Betti numbers of $M$
are $\beta_0=\beta_3=1$ and $\beta_2=\beta_1=2g+1$,
by the Poincar\'e duality with $\Integral/2\Integral$ coefficients.
Therefore, we obtain 
$$\mathrm{TV}_{3,2}(M)=2^{2g},$$
by Lemma \ref{tv_r_odd}.

Since $M$ is homeomorphic to the mapping torus of a periodic surface automorphism of order $d$, 
there is a cyclic cover $\tilde{M}\to M$ of degree $d$, 
and $\tilde{M}$ is a product of a closed orientable surface with a circle.
Since $d$ is odd, the induced homomorphism $H^*(M;\Integral/2\Integral)\to H^*(\tilde{M};\Integral/2\Integral)$
is injective, 
(because of the Poincar\'e duality pairing with $\Integral/2\Integral$ coefficients 
and the isomorphism on the top dimension).
However, $H^1(\tilde{M};\Integral/2\Integral)$ contains no element whose cube is nontrivial,
(by the K\"unneth theorem
which determines the cohomology ring of $\tilde{M}$ over $\Integral/2\Integral$).
It follows that $t^3=0\in H^3(M;\Integral/2\Integral)$ holds for any $t\in H^1(M;\Integral/2\Integral)$.
Moreover, the first Stiefel--Whitney class $w_1$ of $M$ vanishes, as $M$ is orientable.
Therefore, we obtain
$$\mathrm{TV}_{3,1}(M)=2^{2g},$$
by Lemma \ref{tv_r_odd}.
\end{proof}

Under the assumptions of Theorem \ref{main_sfs}, 
we compute the Turaev--Viro invariants in Theorem \ref{main_sfs} as follows.

Suppose that $b^*b_j\equiv \nu_j\bmod a$ holds for some integer $b^*$ coprime to $a$
and some $\nu_j\in\{\pm1\}$, and for all $j\in\{1,\cdots,n\}$.
Then,
\begin{eqnarray*}
\mathrm{TV}_{a,1}(M)
&=& \mathrm{TV}_{a,1}(S^3)\cdot|\tau_a(M)|^2\\
&=& |\tau_a(M)/\tau_a(S^2\times S^1)|^2\\
&=& \left|\frac{a^{g-1}}{2^{n+g-1}a^{n/2}}\cdot \frac{2\cdot a^{n}}{{\sin}^{n+2g-2}(\pi b^*/a)}\right|^2\\
&=&\frac{a^{n+2g-2}}{2^{2n+2g-4}}\cdot\frac{1}{\sin^{2n+4g-4}(\pi b^*/a)}
\end{eqnarray*}
by Lemmas \ref{sfs_rt}, \ref{sfs_nonvanishing}, and \ref{tv_and_rt}.

When $a$ is even, 
we apply Galois conjugacy (transforming $e^{\sqrt{-1}\cdot\pi/a}\mapsto e^{\sqrt{-1}\cdot\pi s/a}$)
for any $s$ coprime to $a$, obtaining
$$\mathrm{TV}_{a,s}(M)=\frac{a^{n+2g-2}}{2^{2n+2g-4}}\cdot\frac{1}{\sin^{2n+4g-4}(\pi b^* s/a)},$$
by Lemma \ref{tv_r_odd}.

When $a$ is odd,
we apply Lemmas \ref{tv_r_odd} and \ref{sfs_cohomology}, obtaining
$$\mathrm{TV}_{a,a-1}(M)=\mathrm{TV}_{3,2}(M)\cdot\mathrm{TV}'_{a,a-1}(M)
=\frac{\mathrm{TV}_{3,2}(M)}{\mathrm{TV}_{3,1}(M)}\cdot\mathrm{TV}_{a,1}(M)=\mathrm{TV}_{a,1}(M)$$
and
$$\mathrm{TV}'_{a,a-1}(M)=\frac{\mathrm{TV}_{a,1}(M)}{\mathrm{TV}_{3,1}(M)}
=\frac{1}{2^{2g}}\cdot\mathrm{TV}_{a,1}(M).$$
Again, 
we apply Galois conjugacy (transforming $e^{\sqrt{-1}\cdot\pi/a}\mapsto e^{\sqrt{-1}\cdot\pi s/a}$
or $e^{\sqrt{-1}\cdot\pi(a-1)/a}\mapsto e^{\sqrt{-1}\cdot\pi s/a}$
depending on $s$ odd or even),
obtaining for any $s$ coprime to $a$
$$\mathrm{TV}_{a,s}(M)=\frac{a^{n+2g-2}}{2^{2n+2g-4}}\cdot\frac{1}{\sin^{2n+4g-4}(\pi b^* s/a)},$$
and if $s$ is even,
$$\mathrm{TV}'_{a,s}(M)=\frac{a^{n+2g-2}}{2^{2n+4g-4}}\cdot\frac{1}{\sin^{2n+4g-4}(\pi b^* s/a)},$$
by Lemma \ref{tv_r_odd}.
This completes the computation of the nonvanishing values in Theorem \ref{main_sfs}.

Suppose the otherwise case. 
Then, for all $r\geq3$ divisible by $a$, we obtain
$$\mathrm{TV}_{r,1}(M)=0,$$
by Lemmas \ref{sfs_rt}, \ref{sfs_vanishing_criterion}, and \ref{tv_and_rt}.
Similarly, we derive $\mathrm{TV}_{r,r-1}(M)=\mathrm{TV}'_{r,r-1}(M)=0$ in this case,
using Lemmas \ref{tv_r_odd} and \ref{sfs_cohomology}.
Finally, by Galois conjugacy (Lemma \ref{tv_r_odd}), we see that 
$\mathrm{TV}_{r,s}(M)=0$ and $\mathrm{TV}'_{r,s}(M)=0$ for any applicable $s$.

This completes the proof of Theorem \ref{main_sfs}.

\section{Examples}\label{Sec-examples}
In this section, we prove Theorem \ref{main_hempel} by exhibiting 
nontrivial Hempel pairs that can or cannot be distinguished by Turaev--Viro invariants.
See Example \ref{example_distinguishable} for our distinguishable ones,
and Example \ref{example_indistinguishable} for our indistinguishable ones.

We need the following lemma for verifying our examples.
A rationality statement would be good enough for our application,
but the integrality here is also well-known to experts.

\begin{lemma}\label{integrality}
Let $S$ be a connected orientable closed surface of genus $g\geq0$,
and $[f]\in\mathrm{Mod}(S)$ be a periodic mapping class of order $d\geq1$.
Then, for any integer $r\geq3$ coprime to $d$, and integer $s$ coprime to $r$,
$$\mathrm{TV}_{r,s}(M_f)\in\Integral,$$
and if $r$ is odd and $s$ is even,
$$\mathrm{TV}'_{r,s}(M_f)\in\Integral,$$
where $M_f$ denotes the mapping torus of $f$.
\end{lemma}

\begin{proof}
By Lemma \ref{mapping_torus_tv}, $\mathrm{TV}_{r,s}(M_f)\in \Complex$
is equal to the trace of the Turaev--Viro TQFT representation 
$\rho^{\mathrm{TV}}_{r,s}([f])\in\mathrm{GL}(\mathcal{Z}^{\mathrm{TV}}_{r,s}(S))$.
If $[f]\in\mathrm{Mod}(S)$ is periodic of order $d$, the eigenvalues 
of $\rho^{\mathrm{TV}}_{r,s}([f])$ are all complex roots of unity of order divisible by $d$.
In particular, $\mathrm{TV}_{r,s}(M_f)$ is an algebraic integer.

On the other hand, entries of $\rho^{\mathrm{TV}}_{r,s}([f])$
lie in the cyclotomic subfield 
$\Rational(e^{\sqrt{-1}\cdot\pi s/r})$ of $\Complex$.
For any roots of unity $\zeta_m,\zeta_n\in\Complex$ of coprime orders $m,n\geq1$,
respectively, it is an elementary fact 
that $\Rational(\zeta_m)\cap\Rational(\zeta_n)$ equals $\Rational$
\cite[Chapter 2, Proposition 2.4]{Washington_book_cyclotomic}.
Applying with $m|d$ and $n=r$ 
(taking $\zeta_n=e^{\sqrt{-1}\cdot\pi s/r}$, 
or if $r$ is odd and $s$ is even,
$\zeta_n=e^{\sqrt{-1}\cdot \pi (r-s)/r}$),
we obtain $\mathrm{TV}_{r,s}(M_f)\in\Rational$.
Together with the algebraic integrality,
we obtain $\mathrm{TV}_{r,s}(M_f)\in\Integral$.

With Lemmas \ref{tv_r_odd} and \ref{sfs_cohomology},
one may derive $\mathrm{TV}'_{r,s}(M_f)\in\Rational$
from $\mathrm{TV}_{r,s}(M_f)\in\Integral$.
To obtain $\mathrm{TV}'_{r,s}(M_f)\in\Integral$,
it is possible to appeal to a similar lemma as Lemma \ref{mapping_torus_tv}
with a TQFT functor associated to $\mathrm{TV}'$.
In fact, this case has been established by Detcherry and Kalfagianni.
We refer to \cite[Corollary 6.1]{DK_monodromy} 
for their proof of this case.
\end{proof}

\begin{example}[Distinguishable pairs]\label{example_distinguishable}
	Let $g\geq0$ and $d=5$ or $d\geq7$ be integers, 
	and $k$ be an integer coprime to $d$.
	Let $M_A$ and $M_B$ be closed Seifert fiber spaces 
	with orientable orbifold base and orientable Seifert fibration.
	We assign their symbols as
	\begin{eqnarray*}
	M_A &\colon &(g;(d,1),(d,1),(d,-1),(d,-1))\\
	M_B &\colon &(g;(d,k^*),(d,k^*),(d,-k^*),(d,-k^*))
	\end{eqnarray*}
	where $k^*$ is an integer satisfing $k^*k\equiv1\bmod p$.
	The $3$--manifold $M_A$ is homeomorphic to the mapping torus
	of some periodic mapping class $[f_A]\in\mathrm{Mod}(S)$ of order $d$,
	where $S$ is a connected orientable closed surface
	of genus $dg+d-2$ (Lemma \ref{mapping_torus_sfs}).
	The $3$--manifold $M_B$ is homeomorphic to the mapping torus
	of the iterate mapping class $[f_B]=[f_A^k]$ (Lemma \ref{mapping_torus_iterate}).
	
	By Theorem \ref{main_sfs},
	we compute 
	\begin{eqnarray*}
	\mathrm{TV}_{d,1}(M_A)&=&\frac{d^{2g+2}}{2^{2g+4}}\cdot\frac{1}{{\sin}^{4g+4}(\pi /d)}\\
	\mathrm{TV}_{d,1}(M_B)&=&\frac{d^{2g+2}}{2^{2g+4}}\cdot\frac{1}{{\sin}^{4g+4}(\pi k/d)}
	\end{eqnarray*}
	The values are equal if and only if $k\equiv\pm1\bmod d$,
	namely, $[f_B]=[f_A]$ or $[f_B]=[f_A^{-1}]$.
	For $k$ other than these values,
	one may also check that $\mathrm{TV}_{d,s}(M_A)\neq \mathrm{TV}_{d,s}(M_B)$
	for any integer $s$ coprime to $d$, and if $d$ is odd,
	$\mathrm{TV}'_{d,s}(M_A)\neq \mathrm{TV}'_{d,s}(M_B)$
	for any even integer $s$ coprime to $d$.
	Under our assumption on $d$,
	such $k$ does exist,
	so $M_A$ and $M_B$ form a nontrivial Hempel pair.
\end{example}

\begin{example}[Indistinguishable pairs]\label{example_indistinguishable}
	Let $g\geq0$ be any integers, and $p\geq5$ be a prime integer,
	and $k$ be an integer coprime to $p$.
	Let $M_A$ and $M_B$ be closed Seifert fiber spaces 
	with orientable orbifold base and orientable Seifert fibration.
	We assign their symbols as
	\begin{eqnarray*}
	M_A &\colon &(g;(p,1),(p,1),(p,-2))\\
	M_B &\colon &(g;(p,k^*),(p,k^*),(p,-2k^*))
	\end{eqnarray*}
	where $k^*$ is an integer satisfing $k^*k\equiv1\bmod p$.
	Obtain 
	the connected orientable closed surface $S$ of genus $pg+(p-1)/2$,
	and the periodic mapping classes $[f_B]=[f_A^k]$ of order $p$,
	similarly as in the previous example 
	(Lemmas \ref{mapping_torus_sfs} and \ref{mapping_torus_iterate}).
		
	Again, $[f_A]$ and $[f_B]$ 
	form a nontrivial Hempel pair exactly when $k\not\equiv\pm1\bmod p$,
	existing under the assumption on $p$.
	
	If $r\geq3$ is divisible by $p$,
	applying Theorem \ref{main_sfs} and Lemma \ref{integrality},
	we see that $\mathrm{TV}_{r,s}(M_A)=\mathrm{TV}_{r,s}(M_B)=0$
	holds for any $s$ coprime to $r$,
	and moreover, if $r$ is odd and $s$ is even,
	$\mathrm{TV}'_{r,s}(M_A)=\mathrm{TV}'_{r,s}(M_B)=0$ also holds.
	
	If $r\geq3$ is not divisible by $p$,
	then it is coprime to $p$.
	By Lemma \ref{integrality}, 
	$\mathrm{TV}_{r,s}(M_A)$ and $\mathrm{TV}_{r,s}(M_B)$ are rational.
	Since $\mathrm{TV}_{r,s}(M_A)$ and $\mathrm{TV}_{r,s}(M_B)$
	are the traces of the Turaev--Viro TQFT representations $\rho^{\mathrm{TV}}_{r,s}$
	of the periodic mapping class $[f_A]$ and $[f_B]=[f_A^k]$, respectively,
	the eigenvalues of $[f_A]$ are roots of unity of order dividing $d$,
	and the eigenvalues of $[f_B]$ are their Galois conjugates
	under the transformation $e^{\sqrt{-1}\cdot2\pi/d}\mapsto e^{\sqrt{-1}\cdot2\pi k/d}$.
	Then by the rationality, we obtain $\mathrm{TV}_{r,s}(M_A)=\mathrm{TV}_{r,s}(M_B)$
	for any $s$ coprime to $r$.
	Moreover, if $r$ is odd and $s$ is even, we apply Lemma \ref{tv_r_odd} and \ref{sfs_cohomology}
	to deduce $\mathrm{TV}'_{r,s}(M_A)=\mathrm{TV}'_{r,s}(M_B)$.
\end{example}

\appendix

\section{Splitting of Turaev--Viro invariants at odd levels}\label{Sec-splitting}
In this appendix section,
we prove the formula in Lemma \ref{tv_r_odd} about $\mathrm{TV}_{r,s}$
when $r$ and $s$ are both odd.
We restate this part as a separate theorem ,
and make a couple of remarks regarding former results.

\begin{theorem}\label{tv_r_odd_s_odd}
Let $M$ be any closed $3$--manifold.
Let $r\geq3$ be an odd integer and $s$ be an integer coprime to $r$.
Adopt Notation \ref{tv_customary}. 

If $s$ is odd, then the following formula holds.
$$\mathrm{TV}_{r,s}(M)=\mathrm{TV}_{3,1}(M)\cdot\mathrm{TV}'_{r,r-s}(M).$$
\end{theorem}

\begin{remark}\label{tv_r_odd_s_odd_remark}\
\begin{enumerate}
\item
Sokolov obtains a canonical splitting of $\mathrm{TV}_{r,s}$
into the sum of three refined invariants \cite{Sokolov_three}.
When $r$ is odd and $s$ is odd, 
one may identify the three refined invariants
(the zeroth, the first, and the second in Sokolov's definition)
as $\mathrm{TV}'_{r,r-s}$, and
$(\mathrm{TV}_{r,s}-\mathrm{TV}_{r,r-s})/2$,
and $(\mathrm{TV}_{r,s}+\mathrm{TV}_{r,r-s})/2-\mathrm{TV}'_{r,r-s}$.
In this case,
the splitting of $\mathrm{TV}_{r,s}$ 
is proportional to the splitting of $\mathrm{TV}_{3,1}$.
Similarly, 
when $r$ is odd and $s$ is even, the splitting of $\mathrm{TV}_{r,s}$
is proportional to the splitting of $\mathrm{TV}_{3,2}$.
Compare \cite[Theorem 1.5]{BHMV_tqft}.
\item
In the same paper, Sokolov quickly points out 
Lemma \ref{sign_change_overall} below,
with assumption of orientability.
See the formula (1) in \cite[Proof of Lemma 2.2]{Sokolov_three}.
\end{enumerate}
\end{remark}

The rest of this section is devoted to the proof of Theorem \ref{tv_r_odd_s_odd}.

Our strategy is to derive needed ingredients from
the proof of Detcherry, Kalfagianni, and Yang
for the case with $r$ odd and $s$ even \cite[Appendix A]{DKY}.
We count sign change from their case 
for individual terms in the defining state-sum expression,
and verify that overall,
the factors $\mathrm{TV}_{3,1}$ and $\mathrm{TV}_{r,s}$
are result of proportional change from factors in their case.

We denote by $\mathrm{ev}_{r,s}\colon\Rational(q^{1/2})\to\Complex$
the evaluation
which assigns 
the abstract root of unity $q^{1/2}$ to be $e^{\sqrt{-1}\cdot\pi s/r}$.

For any odd integer $r\geq3$,
recall that 
$I_r=\{0,1,\cdots,r-2\}$ denotes the set of colors on this level.
It contains the subset of even colors $I'_r=\{0,2,\cdots,r-3\}$
and also $I_3=\{0,1\}$.
For any finite simplicial $3$--complex
$\mathscr{T}=(V,E,F,T)$.
From any coloring $c\colon E\to I_r$,
we obtain a pair of colorings $c_3\colon E\to I_3$ and $c'\colon E\to I'_r$
as follows: For any $e\in E$, we assign $c_3(e)=0$ and $c'(e)=c(e)$ if $c(e)$ is even,
or $c_3(e)=1$ and $c'(e)=r-2-c(e)$ if $c(e)$ is odd.
By observation,
this operation preserves admissible colorings,
and yields a bijective correspondence 
between $\mathcal{A}_r$ and $\mathcal{A}_3\times\mathcal{A}'_r$.

\begin{lemma}\label{x_s_even}
Let $\mathscr{T}=(V,E,F,T)$ be any finite simplicial $3$--complex.
Let $r\geq3$ be an odd integer and $s$ be an integer coprime to $r$.
Adopt Notation \ref{color_terms}.
Identify $\mathcal{A}_r=\mathcal{A}_3\times\mathcal{A}'_r$.
If $s$ is even,
then,
for any $x\in E\sqcup F\sqcup T$ and any $c=(c_3,c')\in \mathcal{A}_r$,
$$\mathrm{ev}_{r,s}\left(|x|_c\right)=\mathrm{ev}_{3,2}\left(|x|_{c_3}\right)\cdot \mathrm{ev}_{r,s}\left(|x|_{c'}\right),$$
and hence,
$$\mathrm{ev}_{r,s}\left(|\mathscr{T}|_c\right)=
\mathrm{ev}_{3,2}\left(|\mathscr{T}|_{c_3}\right)\cdot \mathrm{ev}_{r,s}\left(|\mathscr{T}|_{c'}\right).$$
\end{lemma}

The identities in Lemma \ref{x_s_even} are key to the proof of 
Detcherry, Kalfagianni, and Yang for the $s$ even case \cite[Theorem 2.9]{DKY}.
We refer to \cite[Lemma A.4]{DKY} for the proof.
Note that the identity (A.1) therein essentially relies on the parity assumption of $s$.

For proving the $s$ odd case,
our next few lemmas examine the sign difference between
$\mathrm{ev}_{r,s}(|x|_c)$ and $\mathrm{ev}_{r,r-s}(|x|_c)$,
and between 
$\mathrm{ev}_{r,s}(|\mathscr{T}|_c)$ and $\mathrm{ev}_{r,r-s}(|\mathscr{T}|_c)$.

\begin{lemma}\label{sign_change_x}
Let  $\mathscr{T}=(V,E,F,T)$ be any finite simplicial $3$--complex.
Let $r\geq3$ be an integer and $s$ be an integer coprime to $r$.
Adopt Notation \ref{color_terms}.
Then, for any $x\in E\sqcup F\sqcup T$ and any $c\in\mathcal{A}_r$,
$$\mathrm{ev}_{r,s}\left(|x|_c\right)=(-1)^{\delta(x,c)}\cdot\mathrm{ev}_{r,r-s}\left(|x|_c\right),$$
where $\delta(x,c)\in\Integral$ is assigned as follows.
\begin{enumerate}
\item 
For $x\in E$, having color $i$ under $c$,
$$\delta(x,c)=i.$$
\item
For $x\in F$, having edge colors $(i,j,k)$ under $c$,
$$\delta(x,c)=\frac{i^2+j^2+k^2}2.$$
\item
For $x\in T$, having edge colors $(i,j,k),(i,m,n),(j,m,n),(k,l,n)$
on each face under $c$,
$$\delta(x,c)=i+j+k+l+m+n+\frac{il+jm+kn}2.$$
\end{enumerate}
\end{lemma}

\begin{proof}
We prove for $x\in T$.
The formulas with $x\in E$ and $x\in F$ can be proved by similar means,
and are simpler.
Suppose the tetrahedron $x$ has edge colors $i,j,k,l,m,n$ given by $c$.
Note that the value of the quantum integer only change by a sign 
determined by its parity
$$\mathrm{ev}_{r,s}([w])=(-1)^{w-1}\cdot\mathrm{ev}_{r,r-s}([w]).$$
For any fixed $z$, the total sign change of the summand is $-1$ to the power
\begin{eqnarray*}
& &\frac{(z+1)z}2+\sum_a \frac{(z-T_a)(z-T_a-1)}{2}+\sum_b \frac{(Q_b-z)(Q_b-z-1)}{2}\\
&=&4z^2-2zW+\frac12\cdot\left(\sum_a T_a^2+\sum_b Q_b^2\right)
\end{eqnarray*}
where
$W=i+j+k+l+m+n=\sum_a T_a=\sum_b Q_b$ is an integer.
The first two terms are even integers, 
having no effect to the total change of sign,
and the last term is independent of $z$.
Therefore, we obtain
$$\mathrm{ev}_{r,s}(|x|_c)=(-1)^{\delta(x,c)}\cdot\mathrm{ev}_{r,r-s}(|x|_c).$$
where
$$\delta(x,c)\equiv\frac12\cdot\left(\sum_a T_a^2+\sum_b Q_b^2\right)\mod2.$$
The expression $\sum_a T_a^2+\sum_b Q_b^2$ 
is equal to the sum of all the quadratic monomials in $i,j,k,l,m,n$,
namely, $i^2+ij+\cdots+in+j^2+jk+\cdots+mn$.
We rearrange
$$\sum_a T_a^2+\sum_b Q_b^2=X^2+Y^2+Z^2+XY+XZ+YZ-il-jm-kn$$
where $X=i+l$, $Y=j+m$, and $Z=k+n$.
Note that the parity pattern of $(X,Y,Z)$ may only be
$(0,0,0)$ or $(1,1,1)$, up to permutation of the components.
Indeed, the only possible patterns of $(i_3+l_3,j_3+m_3,k_3+n_3)$
are $(0,0,0)$, $(1,1,1)$, and $(0,2,2)$,
up to permutation of the components,
because of admissible coloring.
So, the part 
$X^2+Y^2+Z^2+XY+XZ+YZ=(X+Y)(X+Z)+Y^2+Z^2$
is congruent to $0$ modulo $4$ if $(X,Y,Z)\equiv(0,0,0)\bmod2$,
or congruent to $2$ modulo $4$ if $(X,Y,Z)\equiv(1,1,1)\bmod2$.
In both cases, we see that 
$(X^2+Y^2+Z^2+XY+XZ+YZ)/2\equiv X+Y+Z\bmod 2$.
This yields
$\delta(x,c)\equiv X+Y+Z-(il+jm+kn)/2\equiv 
i+j+k+l+m+n+(il+jm+kn)/2\bmod2$,
as desired.
\end{proof}

\begin{lemma}\label{sign_change_x_more}
Let  $\mathscr{T}=(V,E,F,T)$ be any finite simplicial $3$--complex.
Let $r\geq3$ be an odd integer and $s$ be an integer coprime to $r$.
Adopt Notation \ref{color_terms}.
Identify $\mathcal{A}_r=\mathcal{A}_3\times\mathcal{A}'_r$.
Let $\delta\colon (E\sqcup F\sqcup T)\times\mathcal{A}_r\to\Integral$ 
be expressed as in Lemma \ref{sign_change_x}.
\begin{enumerate}
\item 
For any $x\in E\sqcup F$ and any $c=(c_3,c')\in\mathcal{A}_r$, 
$$\delta(x,c)\equiv \delta(x,c_3)\mod2,$$
treating $\mathcal{A}_3$ as a subset of $\mathcal{A}_r$.
\item
For $x\in T$, having edge colors $(i,j,k),(i,m,n),(j,m,n),(k,l,n)$
on each face under $c$,
$$\delta(x,c)\equiv \delta(x,c_3)+\lambda(x,c)\mod2,$$
where
$$\lambda(x,c)=\frac{i_3l'+l_3i'+j_3m'+m_3j'+k_3n'+n_3k'}2$$.
\end{enumerate}
\end{lemma}

\begin{proof}
We make use of the relation 
$c(e)=c_3(e)\cdot(r-2-c'(e))+(1-c_3(e))\cdot c'(e)
=(r-2)\cdot c_3(e)-c'(e)-2\cdot c_3(e)\cdot c'(e)$, for any $e\in E$.
Since $r$ is odd and $i',l'$ are even,
we obtain
\begin{eqnarray*}
il&=&((r-2)i_3-i'-2i_3i')\cdot((r-2)l_3-l'-2l_3l')\\
&\equiv& (r-2)^2i_3l_3-(r-2)\cdot(i_3l'+l_3i')\\
&\equiv& i_3l_3+(i_3l'+l_3i')\mod 4
\end{eqnarray*}
and similarly we manipulate $jk$ and $kn$.
Taking the sum, we obtain
$$il+jm+kn\equiv i_3l_3+j_3m_3+k_3n_3+2\cdot\lambda(x,c)\mod4.$$
Moreover, we observe 
$$i+j+k+l+m+n\equiv i_3+j_3+k_3+l_3+m_3+n_3\mod 2.$$
By Lemma \ref{sign_change_x}, the above congruence equalities imply
$\delta(x,c)\equiv \delta(x,c_3)+\lambda(x,c)\bmod2$,
as desired.
\end{proof}

For any $c_3\in\mathcal{A}_3$, 
there is a canonical subsurface $\mathcal{S}(c_3)\subset M$,
in normal position with respect to $\mathscr{T}$,
such that $c_3(e)$
indicates the number of intersection points of any edge
$e\in E$ with $\mathcal{S}(c_3)$.
The subsurface $\mathcal{S}(c_3)$ is formed by taking
one normal disk in each tetrahedron that has nonzero color,
and then taking their union matching the sides.
The types of the normal disks (among four triangular types and three quadrilateral types
for each tetrahedron) are forced by the admissible coloring $c_3$.
The subsurface $\mathcal{S}(c_3)$ is closed, as $M$ does not have boundary.

\begin{lemma}\label{sign_change_overall}
Let $(M,\mathscr{T})$ be 
any triangulated closed $3$--manifold.
Let $r\geq3$ be an odd integer and $s$ be an integer coprime to $r$.
Adopt Notation \ref{color_terms}.
Identify $\mathcal{A}_r=\mathcal{A}_3\times\mathcal{A}'_r$.
Then, for any $c=(c_3,c')\in\mathcal{A}_r$,
$$\mathrm{ev}_{r,s}\left(|\mathscr{T}|_c\right)=
(-1)^{\chi(\mathcal{S}(c_3))}\cdot\mathrm{ev}_{r,r-s}\left(|\mathscr{T}|_c\right),$$
where $\chi(\mathcal{S}(c_3))$ 
denotes the Euler characteristic of the normal subsurface
$\mathcal{S}(c_3)\subset M$ determined by $c_3$.
\end{lemma}

\begin{proof}
For all $x\in E\sqcup F\sqcup T$,
the nonempty intersections $x\cap \mathcal{S}(c_3)$
give rise to a polygonal cell decomposition of $\mathcal{S}(c_3)$.
Denote by $\nu_0,\nu_1,\nu_{2,\triangle},\nu_{2,\square}$ 
the numbers of vertices, edges, triangular normal disks,
and quadrilateral normal disks in $\mathcal{S}(c_3)$, respectively.

Let $\delta\colon (E\sqcup F\sqcup T)\times\mathcal{A}_r\to\Integral$ 
be expressed as in Lemma \ref{sign_change_x}.
For any $x\in E\sqcup F\sqcup T$, it is direct to check
that $\delta(x,c_3)=1$ if and only if $x\cap\mathcal{S}(c_3)$ is 
a vertex or an edge or a quadrilateral normal disk
of $\mathcal{S}(c_3)$, otherwise $\delta(x,c_3)=0$,
treating $\mathcal{A}_3$ as a subset of $\mathcal{A}_r$.
This means 
$\nu_0=\sum_{x\in E}\delta(x,c_3)$, $\nu_1=\sum_{x\in F}\delta(x,c_3)$,
and $\nu_{2,\square}=\sum_{x\in T}\delta(x,c_3)$.
Moreover, 
the relation $3\cdot\nu_{2,\triangle}+4\cdot\nu_{2,\square}=2\cdot\nu_1$
implies that $\nu_{2,\triangle}$ is even,
as $\mathcal{S}(c_3)$ is closed.
By Lemma \ref{sign_change_x_more}, we obtain
\begin{eqnarray*}
\chi(\mathcal{S}(c_3))&=&\nu_0-\nu_1+\nu_{2,\triangle}+\nu_{2,\square}\\
&\equiv& \sum_{x\in E\sqcup F\sqcup T}\delta(x,c_3)\\
&\equiv& \sum_{x\in E\sqcup F\sqcup T}\delta(x,c)+\sum_{x\in T}\lambda(x,c)\mod2.
\end{eqnarray*}

Therefore, 
to derive the asserted formula $\mathrm{ev}_{r,s}(|\mathscr{T}|_c)=
(-1)^{\chi(\mathcal{S}(c_3))}\cdot\mathrm{ev}_{r,r-s}(|\mathscr{T}|_c)$
from Lemma \ref{sign_change_x},
it remains to prove
$$\sum_{x\in T}\lambda(x,c)\equiv0\mod2.$$

To this end, we observe that $\sum_{x\in T}\lambda(x,c)$ is the sum of $c'(e)/2$,
where $e$ ranges over the edges of $x$
whose opposite edge in $x$ meets $\mathcal{S}(c_3)$.
For any edge $e^*\in E$, the link of $e^*$ refers to the union 
of all the opposite edges $e_1,\cdots,e_h$
in all the tetrahedra $t_1,\cdots,t_h\in T$ that contain $e^*$,
denoted as $\mathrm{lk}(e^*)\subset M$.
The link $\mathrm{lk}(e^*)$ is a contractible loop in $M$
(bounding a disk transverse to $e^*$).
On the other hand,
any edge either misses $\mathcal{S}(c_3)$, or meets $\mathcal{S}(c_3)$ at exactly one point.
Then the number of edges in $\mathrm{lk}(e^*)$
that meet $\mathcal{S}(c_3)$ must be even,
and the integer $c'(e^*)/2$ contributes
exactly this even number of times to $\sum_{x\in T}\lambda(x,c)$.
Because $\sum_{x\in T}\lambda(x,c)$ is the total of the contribution
from each edge $e^*\in E$, we conclude
$\sum_{x\in T}\lambda(x,c)\equiv0\bmod2$ as desired,
completing the proof.
\end{proof}

\begin{lemma}\label{T_s_odd}
Let $(M,\mathscr{T})$ be 
any triangulated closed $3$--manifold.
Let $r\geq3$ be an odd integer and $s$ be an integer coprime to $r$.
Adopt Notation \ref{color_terms}.
Identify $\mathcal{A}_r=\mathcal{A}_3\times\mathcal{A}'_r$.
If $s$ is odd,
then,
$$\mathrm{ev}_{r,s}\left(|\mathscr{T}|_c\right)=
\mathrm{ev}_{3,1}\left(|\mathscr{T}|_{c_3}\right)\cdot \mathrm{ev}_{r,r-s}\left(|\mathscr{T}|_{c'}\right).$$
\end{lemma}

\begin{proof}
By Lemma \ref{sign_change_overall}, we obtain
$\mathrm{ev}_{r,s}(|\mathscr{T}|_c)=(-1)^{\chi(\mathcal{S}(c_3))}\cdot\mathrm{ev}_{r,r-s}(|\mathscr{T}|_c)$
and
$\mathrm{ev}_{3,1}(|\mathscr{T}|_{c_3})=(-1)^{\chi(\mathcal{S}(c_3))}\cdot\mathrm{ev}_{3,2}(|\mathscr{T}|_c)$.
Then the asserted identity follows from the $s$ even case (Lemma \ref{x_s_even}).
\end{proof}

To complete the proof of Theorem \ref{tv_r_odd_s_odd},
we observe
\begin{equation}\label{Y_s_both}
\mathrm{ev}_{r,s}\left(Y_r\right)
=
\begin{cases}
\mathrm{ev}_{3,2}\left(Y_3\right)\cdot \mathrm{ev}_{r,s}\left(Y'_r\right)&s\mbox{ even}\\
\mathrm{ev}_{3,1}\left(Y_3\right)\cdot \mathrm{ev}_{r,r-s}\left(Y'_r\right)&s\mbox{ odd}
\end{cases}
\end{equation}
where $Y_r=-(q^{1/2}-q^{-1/2})^2/2r$ and $Y'_r=-(q^{1/2}-q^{-1/2})^2/r$,
(indeed, $\mathrm{ev}_{r,s}(q^{1/2}-q^{-1/2})=-\sqrt{-1}\cdot2\sin(\pi s/r)$).

Let $M$ be any closed $3$--manifold.
If $r$ is odd and $s$ is even,
we obtain
\begin{eqnarray*}
\mathrm{TV}_{r,s}(M)
&=&\mathrm{ev}_{r,s}\left(Y_r\cdot\sum_{c\in\mathcal{A}_r}|\mathscr{T}|_c\right)\\
&=&\mathrm{ev}_{r,s}(Y_r)\cdot\sum_{c\in\mathcal{A}_r}\mathrm{ev}_{r,s}(|\mathscr{T}|_c)\\
&=&\mathrm{ev}_{3,1}(Y_3)\cdot\mathrm{ev}_{r,r-s}(Y'_r)
\cdot\sum_{c_3\in\mathcal{A}_3}\sum_{c'\in\mathcal{A}'_r}
\mathrm{ev}_{3,1}(|\mathscr{T}|_c)\cdot\mathrm{ev}_{r,r-s}(|\mathscr{T}|_{c'})\\
&=&
\mathrm{ev}_{3,1}\left(Y_3\cdot\sum_{c_3\in\mathcal{A}_3}|\mathscr{T}|_c\right)
\cdot
\mathrm{ev}_{r,r-s}\left(Y'_r\cdot\sum_{c'\in\mathcal{A}'_r}|\mathscr{T}|_{c'}\right)\\
&=&\mathrm{TV}_{3,1}(M)\cdot\mathrm{TV}'_{r,r-s}(M)
\end{eqnarray*}
by (\ref{tv_def}), (\ref{tv_refined_def}), (\ref{Y_s_both}), and Lemma \ref{T_s_odd}.
This completes the proof of Theorem \ref{tv_r_odd_s_odd}.

\bibliographystyle{amsalpha}


\end{document}